\documentclass[11pt,reqno]{amsproc}
\usepackage{amsmath,amsfonts,amssymb,amsthm}
\usepackage[abbrev,lite,nobysame]{amsrefs}
\usepackage{mathrsfs}
\usepackage{graphics,graphicx}
\usepackage[dvipsnames]{xcolor}
\usepackage[margin=1in]{geometry}

\usepackage{enumitem}
\usepackage[colorlinks=true, pdfstartview=FitV, linkcolor=BrickRed,citecolor=black, urlcolor=black]{hyperref}

\usepackage[centercolon]{mathtools}
\usepackage{dsfont}
\usepackage{soul}

\DeclareMathAlphabet{\mathup}{OT1}{\familydefault}{m}{n}
\newcommand{\dd}[1]{\mathop{}\!\mathup{d} #1}

\renewcommand{\div}[1]{\mathop{}\!\mathup{div} #1}

\newcommand{\Q}{\mathbb{Q}}
\newcommand{\R}{\mathbb{R}}
\newcommand{\Co}{\mathbb{C}}
\newcommand{\T}{\mathbb{T}}

\newcommand{\Prob}{\mathbb{P}}

\newcommand\e{{\rm e}}

\renewcommand{\Re}{\mathrm{Re}}

\newcommand{\eps}{\varepsilon}

\newtheorem{proposition}{Proposition}[section]
\newtheorem{theorem}{Theorem}[section]
\newtheorem{lemma}{Lemma}[section]

\theoremstyle{definition}

\theoremstyle{remark}
\newtheorem{remark}{Remark}[section]

\newcommand{\TC}{\mathrm{TC}}

\DeclareMathAlphabet{\mathup}{OT1}{\familydefault}{m}{n}

\numberwithin{equation}{section}

\makeatletter
\def\namedlabel#1#2{\begingroup
    #2%
    \def\@currentlabel{#2}%
    \phantomsection\label{#1}\endgroup
}
\makeatother

\makeatletter

\renewcommand\subsubsection{\@startsection{subsubsection}{3}%
\normalparindent{.5\linespacing\@plus.7\linespacing}{-.5em}
{\normalfont\bfseries}}

\def\@tocline#1#2#3#4#5#6#7{\relax
  \ifnum #1>\c@tocdepth % then omit
  \else
    \par \addpenalty\@secpenalty\addvspace{#2}%
    \begingroup \hyphenpenalty\@M
    \@ifempty{#4}{%
      \@tempdima\csname r@tocindent\number#1\endcsname\relax
    }{%
      \@tempdima#4\relax
    }%
    \parindent\z@ \leftskip#3\relax \advance\leftskip\@tempdima\relax
    \rightskip\@pnumwidth plus4em \parfillskip-\@pnumwidth
    #5\leavevmode\hskip-\@tempdima
      \ifcase #1
       \or\or \hskip 1em \or \hskip 2em \else \hskip 3em \fi%
      #6\nobreak\relax
    \dotfill\hbox to\@pnumwidth{\@tocpagenum{#7}}\par
    \nobreak
    \endgroup
  \fi}
\makeatother

\title[Nonlinear instability for 3D MHD around Taylor--Couette]{Nonlinear instability for the 3D MHD equations around the Taylor--Couette flow}

\author[V. Navarro-Fernández]{Víctor Navarro-Fernández}
\address{(VNF) Department of Mathematics, Imperial College London, London, SW7 2AZ, UK}
\email{v.navarro-fernandez@imperial.ac.uk}

\author[D. Villringer]{David Villringer}
\address{(DV) Department of Mathematics, Imperial College London, London, SW7 2AZ, UK}
\email{d.villringer22@imperial.ac.uk}



\date{\today}

\begin{document}

\begin{abstract}
We study the 3D magnetohydrodynamics (MHD) equations in an annular cylinder, perturbed around the explicit steady state given by the 3D Taylor--Couette velocity field and zero magnetic field. Combining a recent linear instability result for the magnetic field with the framework of Friedlander, Pavlovi\'c and Shvydkoy [\emph{Comm.\ Math.\ Phys.}\ 264 (2006), no.\ 2, 335--347], we prove nonlinear instability of the solution around this steady state in $L^p$, for any $p>1$. In particular, our results are, to the best of our knowledge, the first rigorous instability results for 3D MHD without forcing, in which the instability is produced \emph{as a result} of the (exponential) growth of the magnetic field. Furthermore, we offer a mathematical proof of the physically conjectured transfer of energy from the velocity field to the magnetic field in the MHD system.
\end{abstract}

\maketitle

\tableofcontents

\section{Introduction}
Consider the three-dimensional incompressible magnetohydrodynamics (MHD) equations describing the evolution of a magnetic field $B(t,x)\in\R^3$ generated by an electrically charged fluid moving with velocity $u(t,x)\in\R^3$,
\begin{equation}\label{eq:mhd}
\begin{split}
\partial_t u+(u \cdot \nabla) u + \nabla p & = \nu \Delta u+(B \cdot \nabla) B, \\
\partial_t B +(u \cdot \nabla) B-(B \cdot \nabla) u & = \eps \Delta B, \\
\div(u) & = 0, \\
\div(B) & = 0,
\end{split} 
\end{equation}
with $(t,x)\in (0,\infty)\times\Omega$.
Here $p(t,x)$ denotes the pressure of the fluid, $\nu>0$ its kinematic viscosity, and $\eps>0$ the magnetic resistivity. The fluid is contained in a periodic annular cylinder, that we denote by $\Omega\subset\R^3$, which in cylindrical coordinates is defined in terms of the inner and outer radii $0<R_1<R_2$ by
\begin{equation}\label{eq:domain}
\Omega = \{(r,\theta,z)\in (0,\infty)\times\T\times\T: R_1\leq r\leq R_2\}.  
\end{equation}
The boundaries of this annular cylinder are physical---thus, no fluid can go across---rotate with a constant speed, and move vertically at a constant relative velocity. Upon picking a suitable reference frame, we can always assume that the cylinder corresponding to the boundary $r=R_1$ is stationary, whereas the cylinder corresponding to $r=R_2$ is moving. This, on the one hand, translates into the following boundary conditions for the velocity field
\begin{equation}\label{eq:BC-u}
u\cdot \hat n|_{\partial \Omega}=0, \quad u_\theta|_{r=R_1}=u_z|_{r=R_1}=0, \quad u_\theta|_{r=R_2}=\beta_1, \quad u_z|_{r=R_2}=\beta_2,
\end{equation}
where $\beta_1,\beta_2\in \R\setminus\{0\}$, and $\hat n$ denotes the unit normal outwards vector to the boundary $\partial \Omega$. On the other hand, for the magnetic field we consider the boundary conditions of a perfectly conducting wall, which are modelled by
\begin{equation}\label{eq:BC-B}
B \cdot \hat n|_{\partial \Omega}=0, \quad \hat n \times (\nabla \times B)|_{\partial \Omega}=0.
\end{equation}
Provided that we take some initial configuration $(u_0,B_0)$ in a suitable Banach space, e.g.\ $L^p_\omega(\Omega)$---the space of divergence-free vector valued functions in $L^p(\Omega)$---we find that \eqref{eq:mhd}--\eqref{eq:BC-B} defines a locally well-posed Cauchy problem, see \cites{DuvautLions72,SermangeTemam1983}. For further information about the MHD equations and physically relevant boundary conditions, we refer to the monographs \cites{roberts1967,Davidson01}.

In this setting, we introduce a special stationary solution to equations \eqref{eq:mhd} with the boundary conditions \eqref{eq:BC-u}--\eqref{eq:BC-B}. We denote this solution by $(u_\TC,0)$, it consists of the trivial magnetic field $B=0$, together with the three-dimensional \emph{Taylor--Couette} radial vector field, which in cylindrical coordinates has the form
\begin{equation}\label{eq:TC}
    u_\TC(r) = \left(\frac{a_1}{r}+a_3r\right)\hat\theta + \left(a_2\log(r)+a_4\right)\hat z.
\end{equation}
The coefficients $a_1,a_2,a_3,a_4\in\R$, with $a_1,a_2\neq 0$, are chosen to satisfy the boundary conditions \eqref{eq:BC-u}. Given any viscosity parameter $\nu>0$, the Taylor--Couette velocity field \eqref{eq:TC} is an explicit solution to the Navier--Stokes equations in $\Omega$ with boundary conditions \eqref{eq:BC-u}, without forcing, and with a radial pressure term $p_\TC$ given by
\[
\partial_r p_\TC(r) = \frac{1}{r} (u_\TC)_\theta^2 = \frac{a_1^2}{r^3} + a_3^2r.
\]

The main goal of this note is to show that the MHD equations \eqref{eq:mhd} are \emph{nonlinearly unstable} around the steady state $(u_\TC,0)$ in $L^p$ for every $p\in (1,\infty)$. Nonlinear instability in this context refers to the existence of an initial configuration $(u_0,B_0)$ arbitrarily close to the steady state in a suitable Banach space, and a time $t_\star>0$ at which the corresponding solution $(u(t_\star),B(t_\star))$ has drifted a distance of order $1$ from the steady state. In our particular case we can allow for perturbations around $(u_\TC,0)$ that are small in any Sobolev space $W^{s,p}$ with $s\geq 0$ and $p\in (1,\infty)$, and we find a solution $(u,B)$ that departs from the steady state in $L^p$.
Our line of reasoning exploits the recent linear instability result for the magnetic field \cite{NFVillringer}---see details in Section \ref{s:linear}---and adapts the arguments from Friedlander, Pavlovi\'c and Shvydkoy \cite{Friedlander_Pavlović_Shvydkoy_2006} to the full MHD equations \eqref{eq:mhd} with the physically relevant boundary conditions \eqref{eq:BC-u}--\eqref{eq:BC-B}.

Nonlinear instability results for \eqref{eq:mhd} around any linearly unstable flow for the Navier--Stokes equations (combined with the trivial background magnetic field $B=0$) may immediately be deduced from corresponding linear instability via \cite{Friedlander_Pavlović_Shvydkoy_2006}. More popular in the literature seems to be the opposite question: are there stability thresholds for the MHD equations when perturbed around some steady state? This has been typically answered for the \emph{Couette flow}. In two-dimensions and Cartesian coordinates, the Couette flow in $\T\times\R$ has the form $u = (y,0)$. When combined with the constant magnetic field $B = (\beta,0)$, $\beta\neq 0$, it has been shown to be a steady state of the 2D MHD equations that is stable in $L^2$ under sufficiently small perturbations in $H^s$, with $s>0$ big enough \cites{Dolce24,Knobel25,WZ24+,JRW25}. Similar results but in Gevrey spaces exist if one of the diffusion parameters vanishes, i.e.\ $\nu=0$ \cites{KZ23,ZZ24}, $\eps=0$ \cite{DKZ24+}, or $\nu=\eps=0$ \cite{Knobel25b}. For the 3D MHD equations in $\T\times\R\times\T$, one can find analogous thresholds for the Couette flow $u=(y,0,0)$ together with a constant background magnetic field of the form $B=\alpha(\sigma,0,1)$, for $\alpha\in\R$ sufficiently large and $\sigma\in \R\setminus\Q$ satisfying some Diophantine condition---which has been shown to be stable in $L^2$ for certain regimes of the diffusion parameters $\nu,\eps>0$ \cites{Liss20,RZZ25,WXZ25+}. The stability of static steady states $u=0$ has also been studied together with a constant nontrivial magnetic fields in $\R^2$ or $\R^3$ \cites{BardosSulemSulem88,HeXuYu18}, or with a monotone shear magnetic field $B=(b(y),0)$ in $\T\times[0,1]$ \cite{RenZhao17}.

Closer to the purposes of this note, we mention the works of G\'erard-Varet and Rousset \cites{GerardVaret05,GerardVaretRousset07}, who proved nonlinear instability in $L^2$ for the 3D MHD equations that includes a forcing in the Navier--Stokes equations. They made use of the dynamo mechanism through which the action of a velocity field $u$ makes the magnetic field $B$ grow exponentially. We draw special attention to the paper \cite{GerardVaretRousset07}, where velocity fields of a similar form to those in \cite{NFVillringer} are considered. In particular, the authors define a sequence of vector fields $u^\kappa$ that converge formally to the discontinuous Ponomarenko velocity field as $\kappa \to 0$, see \cites{gilbert1988, NFVillringer}. Via the method of matched asymptotic expansions, they show that the MHD equations near this (forced) steady state are nonlinearly unstable, in the sense that a perturbation of size $O(\kappa^p)$, with $p>0$ and $\kappa\ll1$, grows to $O(1)$ in times $O(|\log(\kappa)|)$.

\subsection*{Notation}

We use the symbol $a\lesssim b$ to denote that there exists a constant $C>0$, independent of all the relevant quantities, such that $a\leq Cb$. Moreover, we use the symbol $a\sim b$ to denote that $a\lesssim b$ and $b\lesssim a$. In many different estimates we will use a generic constant $C>0$ that might change from line to line.
Given a Banach space $X$ of vector valued functions $w:\Omega\to\R^3$, we define
\[
X_\omega = \{ w\in X : \div(w) = 0\}.
\]
Given a function space $X$ and a vector valued function $w=(w_1,\hdots, w_n)$ such that $w_i\in X$ for all $1\leq i\leq n$, we abuse the notation and write $w\in X$.

\subsection{Main result}

As in \cites{GerardVaret05,GerardVaretRousset07}, we derive a nonlinear instability result via the dynamo mechanism, which provides exponential growth for the linearised magnetic field. In this regard, we write an equation for the perturbation $(v,B)=(u-u_\TC,B)$, which has the form
\[
\partial_t\begin{pmatrix}
    v\\
    B
\end{pmatrix} = A\begin{pmatrix}
    v\\
    B
\end{pmatrix} + \text{nonlinear terms}.
\]
Here, $A = A_1\oplus A_2$ defines a linear operator that can be written as the direct sum of the linearised Navier--Stokes operator $A_1$,
\[
A_1 v = \Prob(-(u_\TC\cdot\nabla) v - (v\cdot\nabla) u_\TC + \nu\Delta v),
\]
where $\Prob$ denotes the Leray projector onto the space of divergence-free vector fields, and the kinematic dynamo operator $A_2$,
\[
A_2 B = -(u_\TC\cdot\nabla) B + (B\cdot\nabla) u_\TC + \eps\Delta B.
\]
If we can show that either $A_1$ or $A_2$ admits a growing mode---as \cite{NFVillringer} ensures for kinematic dynamo operator $A_2$---we say that the corresponding equations are \emph{linearly unstable}. Then, via the arguments in \cite{Friedlander_Pavlović_Shvydkoy_2006}, we can upgrade the instability for the full nonlinear equations.

\begin{theorem}\label{theorem}
Let $u_\TC$ denote the Taylor--Couette velocity field \eqref{eq:TC}.
For any $s\geq 0$, $\nu>0$, and any $\eps>0$ small enough, there exist $\chi,c>0$ so that the following holds true: 
For all $\delta>0$, there exists an initial configuration $(u_0,B_0)\in W^{s,p}$ with 
\[
\left\|\begin{pmatrix}
    u_0\\
    B_0
\end{pmatrix}-\begin{pmatrix}
u_\TC\\
0
\end{pmatrix}\right\|_{W^{s,p}} \leq \delta,
\]
and a time $t_\star =c|\log(\delta)|$, so that the solution to the MHD equations \eqref{eq:mhd} satisfies
\begin{equation*}
\left\|\begin{pmatrix}
u(t_\star)\\
B(t_\star)
\end{pmatrix}-\begin{pmatrix}
u_\TC\\
0
\end{pmatrix}
\right\|_{L^p} \geq \chi.
\end{equation*}
Furthermore, if $\nu\gg\|u_\TC\|_{W^{1,\infty}}$, there holds
$\|B(t_\star)\|_{L^p}\geq \chi$.
\end{theorem}

Theorem \ref{theorem} shows that the MHD equations \eqref{eq:mhd} with sufficiently small magnetic resistivity $0<\eps\ll 1$ and boundary conditions given by \eqref{eq:BC-u} and \eqref{eq:BC-B} are nonlinearly unstable in $L^p$, for any $1<p<\infty$, around the Taylor--Couette velocity field \eqref{eq:TC} and the trivial magnetic field $B=0$.

\begin{remark}
From the proof of Theorem \ref{theorem} in Section \ref{s:proof}, it will become clear that the actual behaviour of $(u(t_\star),B(t_\star))$ depends primarily on the spectral picture of the linearised operator $A=A_1\oplus A_2$. In fact, $(u(t_\star),B(t_\star))$ will look like the eigenmode of $A$ associated to the eigenvalue with largest real part, perturbed by the nonlinear terms from the mild formulation. In particular, if the operator $A_1$ is linearly stable, which from energy estimates can be seen to occur for instance when $\nu \gg \|u_{\TC}\|_{W^{1,\infty}}$, then $\|B(t_\star)\|_{L^p} \sim \chi $. In other words, provided $\nu$ is not too small, we see that the nonlinear instability develops due to a transfer of energy from the velocity field to the magnetic field, as conjectured in the physical literature, see e.g.\ \cites{BrandenburgSubramanian2005,Moffatt1978,RuedigerHollerbach2006,ChildressGilbert}.
\end{remark}

An interesting follow-up question would be to study the dependence of the coefficients $\chi$ and $t_\star$ from Theorem \ref{theorem} on the resistivity $\eps$, in the case where $\nu \gg \|u_{\TC}\|_{W^{1,\infty}}$. Our methods show that  \[
t_\star \sim \frac{1}{\lambda} \log \left(\frac{\chi}{\delta}\right),
\]
where $\lambda\gtrsim \eps^{1/3}$ denotes the eigenvalue corresponding to the fastest growing mode of the linear operator $A_2$. However, the value of $\chi$ depends on the constants appearing in Lemma \ref{lemma:abstract semigroup lemma}. Quantifying these in $\eps$ would necessitate precise resolvent estimates for the linearised operator $A_2$ in $L^p$, which do not appear to follow from our analysis in \cites{NFVillringer}.

\section{Linear instability and the kinematic dynamo}\label{s:linear}

The main cornerstone of our proof is the linear instability result that was recently obtained in \cite{NFVillringer}. We shall dedicate this section to a brief overview of the constructions and theorems from that paper in order to yield a better understanding of the result. This work is concerned with spectral instability of the so-called \emph{kinematic dynamo equations}, namely the following linear problem,
\begin{equation}
    \begin{split}\label{eq:dynamo equation}
        \partial_t B + (\bar u\cdot\nabla)B - (B\cdot\nabla)\bar u & = \eps \Delta B, \\
        \div(B) & = 0,
    \end{split} 
\end{equation}
with $(t,x) \in (0,\infty) \times \Omega$, endowed with the perfectly conducting boundary conditions $B\cdot\hat n|_{\partial\Omega} = 0$, $\hat n\times(\nabla\times B)|_{\partial\Omega} = 0$, and an initial datum $B(0,\cdot) = B_0$.
As before, $\hat n$ denotes the unit outwards pointing normal vector to $\partial\Omega$.
In this problem, the magnetic field $B$ is \emph{passively} transported by a given autonomous and divergence-free velocity field $\bar u$ of the form
\[
\bar u(r)=r\Omega(r)\hat\theta+U(r)\hat z,
\]
with $U,\Omega$ smooth radial functions. This type of vector fields were introduced by Ponomarenko in the context of the dynamo problem \cites{Ponomarenko73,gilbert1988}, which asks whether it is possible to find examples of velocity fields $\bar u$ such that the magnetic energy associated to the linear problem \eqref{eq:dynamo equation} grows exponentially in time,
\begin{equation}\label{eq:exp-growth}
\|B(t)\|_{L^2}\geq C\e^{\lambda t}\|B_0\|_{L^2},
\end{equation}
for some constants $C,\lambda>0$. 

The results in \cite{NFVillringer} can be applied to the case of $\Omega = [R_1,R_2] \times \mathbb{T}\times \mathbb{T}$ as defined in \eqref{eq:domain} with the Taylor--Couette velocity field $\bar u = u_\TC$ \eqref{eq:TC}. Therefore we see that the equations \eqref{eq:dynamo equation} can be written in terms of the kinematic dynamo operator $A_2$ introduced in the previous section: $\partial_t B = A_2B$. In this setting we find the following proposition.

\begin{proposition}[\cite{NFVillringer}*{Theorem 1.2, Lemma 3.5}]\label{prop:dynamo}
    For every $\eps>0$ sufficiently small, there exist a $C^2$ divergence-free vector field $B_0:\Omega\to\R^3$ and a constant $\lambda\sim\eps^{1/3}>0$ so that
    \[
    B(t,r,\theta,z) = \e^{\lambda t}B_0(r,\theta,z)
    \]
    defines a solution to the linear problem \eqref{eq:dynamo equation} with Taylor--Couette velocity field $\bar u = u_\TC$ \eqref{eq:TC} and perfectly conducting boundary conditions \eqref{eq:BC-B} in the annular cylinder $\Omega$.
\end{proposition}

Proposition \ref{prop:dynamo} proceeds via establishing the existence of an unstable eigenmode $B_0\in L^2_\omega(\Omega)$ for the kinematic dynamo operator $A_2$. The proof of this result revolves around detailed resolvent estimates that can be carried out via the introduction of tailor-made Green's functions, used to invert $A_2$ near the unstable eigenvalue $\lambda$. 

The aim of this note will be to ``upgrade'' the linear instability result of Proposition \ref{prop:dynamo} to a nonlinear instability result for the full nonlinear 3D MHD equations.

%\section{Nonlinear instability: proof of Theorem \ref{theorem}}

\section{Setup and technical lemmas}

We consider a divergence-free perturbation $(v,B)$ around the steady state $(u_\TC, 0)$ with Dirichlet boundary conditions for the velocity field $v|_{\partial\Omega}=0$, and perfectly conducting boundary conditions for the magnetic field $B\cdot\hat n|_{\partial\Omega} = \hat n\times(\nabla\times B)|_{\partial\Omega} = 0$. Let $\Prob$ denote the Leray projector operator onto the space of divergence-free vector fields. Then, the perturbation satisfies the equation
\begin{equation}\label{eq:eq-perturbation}
\begin{split}
\partial_t \begin{pmatrix}
v\\
B
\end{pmatrix}
=A\begin{pmatrix}
v\\
B
\end{pmatrix}
+\begin{pmatrix}
-N(v,v)+N(B,B)\\
M(v,B)
\end{pmatrix},
\end{split}
\end{equation}
where $A$ denotes the following linear operator
\begin{equation}\label{eq:A}
A\begin{pmatrix}
v\\
B
\end{pmatrix}=\begin{pmatrix}
\mathbb{P}(-(u_\TC\cdot \nabla) v-(v \cdot \nabla) u_\TC+\nu \Delta v) \\
- (u_\TC\cdot\nabla)B + (B\cdot\nabla)u_\TC + \eps\Delta B
\end{pmatrix},
\end{equation}
The nonlinear part is given by $N(v,v)=\mathbb{P} (\nabla \cdot (v\otimes v))$ and $M(v,B)=\nabla \times (v \times B)$.  In particular, the operator $A$ admits a rather simple, diagonal structure, in the sense that it may be written as $A=A_1 \oplus A_2$, where $A_1$ is the linearised Navier--Stokes operator with homogeneous Dirichlet boundary conditions, as studied already in \cite{Friedlander_Pavlović_Shvydkoy_2006}, and $A_2$ is the kinematic dynamo operator with perfectly conducting boundary conditions around the Taylor--Couette flow, for which Proposition \ref{prop:dynamo} yields an unstable eigenmode. This diagonal structure simplifies certain computations substantially---in particular, for any holomorphic $f :\Co \to \Co$ for which the functional calculus for sectorial operators is well defined, $f(A)$ is nothing but the operator $f(A_1) \oplus f(A_2)$. That being said, the following result shall be the main technical ingredient of this section.
\begin{lemma}
\label{lemma:abstract semigroup lemma}
Let $\eta>0$, and let $\lambda$ be the largest real part of any eigenvalue of the linear operator $A$ defined in \eqref{eq:A}. Set $A_\eta=A-\lambda-\eta$. Then, 
\begin{itemize}
    \item for any $\alpha \in (0,1)$, $p>1$ there exists a constant $C>0$ so that 
    \begin{equation*}
    \|A_{\eta}^\alpha \e^{A_\eta^\alpha t}\|_{L^p \to L^p} \leq \frac{C}{t^\alpha};
    \end{equation*}
    \item for any $\alpha\in (1/2,1)$ and any $q  \leq 3p/(3-(2\alpha-1)p)$, there exists a constant $C>0$ so that 
    \begin{equation*}
    \|A^{-\alpha}_{\eta}  \nabla  H\|_{L^q} \leq C\|H\|_{L^{p}}.
    \end{equation*}
    Here we abuse notation and write $\nabla H=\begin{pmatrix}
    \nabla H_1\\
    \nabla H_2
    \end{pmatrix}$
    for some vector fields $H_1, H_2$.
\end{itemize}

\end{lemma}

The usefulness of this technical lemma will become apparent in the proof of Theorem \ref{theorem} in Section \ref{s:proof}. In particular, within the mild formulation, when the semigroup $\e^{A(t-\tau)}$ for $0\leq \tau\leq t$ acts on the nonlinear terms, we employ the identity
\[
\e^{A(t-\tau)} = \e^{(\lambda+\eta)(t-\tau)}A_\eta^\alpha \e^{A_\eta(t-\tau)}A_\eta^{-\alpha}
\]
to control the corresponding $L^p$ norms.

Via the diagonal decomposition $A=A_1 \oplus A_2$, it suffices to prove the result separately for $A_1$, $A_2$. Furthermore, the result for $A_1$ has already been shown in \cite{Friedlander_Pavlović_Shvydkoy_2006}*{Lemmas 3.1, 3.2}, so it remains to cover the $B$ component, i.e.\ to prove the corresponding result for the operator $A_2$. 

Analysing this operator will turn out to be somewhat non-trivial, since it is inherently a perturbation of the \emph{Hodge Laplacian}, rather than the Dirichlet Laplacian. In particular, this operator is no longer self-adjoint on $L^2$, unless it is considered on the space of divergence free vector fields, where it acts as $\Delta=-\nabla \times \nabla \times$. As such, we take $A_2$ to have domain 
\begin{equation*}
\mathcal{D}(A_2)=\{B \in W^{2,p}_\omega(\Omega):B \cdot \hat n|_{\partial \Omega}=\hat n \times (\nabla \times B)|_{\partial \Omega}=0\}.
\end{equation*}
In order to prove Lemma \ref{lemma:abstract semigroup lemma}, the goal will then be to show that $A_2$ behaves ``as though it was a Dirichlet Laplacian''. In particular, its fractional powers should generate suitable fractional Sobolev spaces.

It is worth mentioning that from a differential geometry perspective, the Hodge Laplacian is defined in terms of the exterior derivative $d$ and its formal adjoint $\Delta$ as $\Delta B = (d\delta + \delta d)B$, where now $B$ denotes a $k-$form in the Riemannian manifold $\Omega$. In this formulation, one finds the following natural set of boundary conditions, so-called \emph{absolute} or \emph{Neumann-type},
\[
\iota_{\hat n} B|_{\partial\Omega} = 0, \quad \iota_{\hat n}dB |_{\partial\Omega} = 0,
\]
where the symbol $\iota_{\hat n}$ denotes the inner product with the unit vector $\hat n$ that is normal to $\partial\Omega$. In particular, if $B$ is a $1-$form like a three dimensional vector field, these boundary conditions coincide with the perfectly conducting setting introduced in \eqref{eq:BC-B}. For further information about the Hodge Laplacian in Riemannian manifolds we refer to \cite{GueriniSavo03}.

To start, let us recall classical results on the $L^p$ theory of systems of elliptic equations. In particular, we shall briefly mention the classical work of \cite{Agmon_Douglis_Nirenberg_1964}, which is concerned with elliptic systems of differential equations satisfying so called \emph{complementing boundary conditions}. The verification of these conditions for the Laplacian is a rather simple, although somewhat arduous endeavour. As a consequence, we shall state here a list of properties of the Hodge Laplacian with perfectly conducting boundary conditions. 
\begin{lemma}[Properties of the Hodge Laplacian (\cites{Agmon_Douglis_Nirenberg_1964, Seeley_1971, Seeley_1972})]
\label{properties of Hodge Laplacian}
Let $\Omega \subset \mathbb{R}^3$ be a smooth domain, and consider the operator $\Delta$ on $L^p(\Omega)$, with domain 
\[
\mathcal{D}(\Delta)=\{B \in W^{2,p}(\Omega):B \cdot \hat n|_{\partial \Omega}=\hat n \times (\nabla \times B)|_{\partial \Omega}=0\}.
\]
Then, the following statements hold true.
\begin{itemize}
    \item There exists a constant $K_1>0$ so that for any $F \in L^p(\Omega)$ such that $\Delta B=F$, $B \in \mathcal{D}(\Delta)$ there holds 
    \begin{equation}
    \label{eq:hodge laplacian sobolev}
    \|B\|_{W^{2,p}(\Omega)} \leq K_1\left(\|F\|_{L^p(\Omega)}+\|B\|_{L^p(\Omega)}\right).
    \end{equation}
    \item For all $c >0$ large enough, and for any $\theta \in [0,1]$, the interpolation space $\mathcal{D}((-c+\Delta)^\theta)$ is well-defined, equal to a subspace of $W^{2\theta,p}(\Omega)$, and there exists a constant $K_2>0$ so that for all $B \in \mathcal{D}((-c+\Delta)^\theta)$ it holds
    \begin{equation}
    \label{eq:interpolation inequality}
    \|B\|_{W^{2\theta,p}(\Omega)} \leq K_2 \|(-c+\Delta)^\theta B\|_{L^p(\Omega)}.
    \end{equation}
\end{itemize}
\end{lemma}

\begin{proof}
The estimate \eqref{eq:hodge laplacian sobolev} follows immediately from \cite{Agmon_Douglis_Nirenberg_1964}*{Theorem 10.5}. To prove \eqref{eq:interpolation inequality}, we first note that by \cite{Seeley_1971}, we know that $\mathcal{D}((-c+\Delta)^\theta)=[L^p(\Omega),\mathcal{D}(\Delta)]_\theta$, and from a reading of the proof of \cite{Seeley_1971}*{Theorem 3}, there exists a constant $C>0$ so that $\|B\|_{\theta} \leq C\|(-c+\Delta)^\theta B\|_{L^p(\Omega)}$, where $\|\cdot \|_{\theta}$ denotes the interpolation norm of $[L^p(\Omega),\mathcal{D}(\Delta)]_\theta$, see e.g.\ \cite{Calderón_1964}. Furthermore, from \cite{Seeley_1972}, we further deduce the existence of a constant $C_2>0$ so that for all $B \in [L^p(\Omega),\mathcal{D}(\Delta)]_\theta$, there holds $\|B\|_{W^{2\theta,p}(\Omega)} \leq C_2\|B\|_{\theta}$. Hence, \eqref{eq:interpolation inequality} follows.
\end{proof}

We remark that in \cites{Agmon_Douglis_Nirenberg_1964,Seeley_1971,Seeley_1972} very general sets of boundary conditions are considered---so-called \emph{complementary} \cite{Agmon_Douglis_Nirenberg_1964}*{Section 2} or \emph{normal} \cite{Seeley_1972}*{Definition 3.1} boundary conditions---which in particular include the perfectly conducting case \eqref{eq:BC-B}. These are given locally via an algebraic relation, and ensure that the boundary value problem for the elliptic operator is well-posed. In addition, we want to highlight that the perfectly conducting boundary conditions for divergence-free vector fields lead to a dissipative realization of the (vector) Laplacian, since we have the identity
\[
\int_\Omega \Delta B\cdot B\dd x = -\int_\Omega |\nabla\times B|^2\dd x.
\]
In particular, from this expression it is clear that the Laplacian is a negative operator, and as noted already in \cite{SermangeTemam1983}, $\|\nabla \times B\| \sim \|B\|_{\dot H^1}$ for divergence-free vector fields. It is clear that the kinematic dynamo operator admits a compact resolvent on the space of divergence-free vector fields satisfying the perfectly conducting boundary conditions. Since the same property holds true for $A_1$, it follows that all eigenvalues of $A$ are isolated and discrete.
With these preliminaries out of the way, we now turn to the proof of Lemma \ref{lemma:abstract semigroup lemma}.

\subsection{Proof of Lemma \ref{lemma:abstract semigroup lemma}}

The proof will be split into multiple shorter results, the first and most essential of which is the following.

\begin{lemma}
\label{lemma:analytic semigroup}
Let $1<p<\infty$, and consider $A_2B=\nabla \times (u_\TC \times B)+\eps \Delta B$ on $L^p_\omega(\Omega)$, with domain 
$$
\mathcal{D}(A_2)=\{B \in W^{2,p}_\omega(\Omega):B \cdot \hat n|_{\partial \Omega}=\hat n \times (\nabla \times B)|_{\partial \Omega}=0\}.
$$
Then, $A_2$ generates an analytic semigroup on $L^p_{\omega}(\Omega)$.
Furthermore, for any $\mu \in \rho(A_2)$, where $\rho(A_2)$ denotes the resolvent set, there exists a constant $C>0$ depending only on $\mu$, so that for all $B \in W^{1,p}(\Omega)$ there holds 
\begin{equation}
\|(A_2-\mu)^{-1}\nabla B\|_{W^{1,p}(\Omega)} \leq C\|B\|_{L^p(\Omega)}.
\end{equation}

\end{lemma}
\begin{proof}
We begin by noting that by \cite{Miyakawa_1980}*{Theorem 3.10}, the operator $\Delta$ on $L^p_{\omega}$, with domain $\mathcal{D}(A_2)$ generates an analytic semigroup. Hence, it remains to show that the addition of the terms $\nabla \times (u_\TC \times \cdot)$ does not disturb this property. To do so, note that certainly for any $B \in \mathcal{D}(A_2)$ there holds that 
$$
\|\nabla \times (u_\TC \times B)\|_{L^p(\Omega)}\leq C(\|u_\TC\|_{W^{1,\infty}})\|B\|_{W^{1,p}}.$$
But now, by the Gagliardo--Nirenberg inequality, we can write 
$$
\|\nabla B\|_{W^{1,p}(\Omega)}\leq C\|B\|_{W^{2,p}(\Omega)}^{1/2}\|B\|_{L^p(\Omega)}^{1/2}.$$
Hence, for any $\eta>0$, there exists $C(\eta)>0$ so that for all $B \in W^{2.p}(\Omega)$ we can bound
\begin{equation}
\label{eq:dynamo bound Laplacian}
\|\nabla \times (u_\TC \times B)\|_{W^{1,p}(\Omega)}\leq \eta \|B\|_{W^{2,p}(\Omega)}+C(\eta)\|B\|_{L^p(\Omega)}.
\end{equation}
Next, from \eqref{eq:hodge laplacian sobolev} we deduce the existence of a \emph{uniform} constant $K>0$ so that if $\Delta B=F,
$ complemented with perfectly conducting boundary conditions, there holds 
\begin{equation*}
\|B\|_{W^{2,p}(\Omega)}\leq K\left (\|F\|_{L^p}+\|B\|_{L^p}\right ).
\end{equation*}
In other words, we have that for any $B \in \mathcal{D}(A_2)$
\begin{equation*}
\|B\|_{W^{2,p}(\Omega)}\leq K\left (\|\Delta B\|_{L^p}+\|B\|_{L^p}\right ).
\end{equation*}
Hence, combining this with \eqref{eq:dynamo bound Laplacian}, for any $\eta>0$ there exists $C(\eta)>0$ so that for all $B \in \mathcal{D}(A_2)$, 
\begin{equation}
\label{eq:relative perturbation}
\|\nabla \times (u_\TC \times B)\|_{L^p(\Omega)}\leq \eta \|\Delta B\|_{L^p(\Omega)}+C(\eta)\|B\|_{L^p(\Omega)}.
\end{equation}
Consequently, by classical functional analysis results, see for instance \cite{Kato}, the operator 
\[
A_2=\eps \Delta +\nabla \times (u_\TC \times \cdot)
\]
is closed when given the domain $\mathcal{D}(A_2)$, and generates an analytic semigroup (see also Lemma \ref{lemma:hairer}), completing the proof of the first claim.

To show the second claim of the lemma, we once again use the inequality \eqref{eq:relative perturbation} to note that for any $\eta>0$, there exists $C(\eta)>0$ so that for all $B \in \mathcal{D}(A_2)$ there holds 
\begin{equation*}
\|\nabla  \times (u \times B)\|_{L^p(\Omega)} \leq \eta \|\Delta B\|_{L^p(\Omega)}+C(\eta)\|B\|_{L^p(\Omega)}.
\end{equation*}
In particular, taking $\Delta \mapsto \eps \Delta$, for $\mu>0$ very large we find that
\begin{align*}
\|\nabla \times (u\times (\eps \Delta-\eps\mu)^{-1}B)\|_{L^p} & \leq \eps \eta \|\Delta (\eps \Delta-\eps \mu)^{-1}B\|_{L^p}+C(\eps \eta)\|(\eps \Delta-\eps \mu)^{-1}B\|_{L^p}\\
&\leq \eta \|B\|_{L^p}+\eta \mu\|(\Delta-\mu)^{-1}B\|_{L^p} +\eps^{-1}C(\eps \eta)\|(\Delta-\mu)^{-1}B\|_{L^p}\\
&\leq \eta (1+C)\|B\|_{L^p(\Omega)}+\eps^{-1}C(\eps \eta))C\mu^{-1}\|B\|_{L^p},
\end{align*}
where we used the fact that, since $\Delta$ generates an analytic semigroup, there exists $C>0$ so that for all $\mu>0$ large enough it holds $\|( \Delta-\mu)^{-1}\|\leq C\mu^{-1}$. Therefore, taking first $\eta>0$ sufficiently small, and then $\mu$ sufficiently large, we get that 
\begin{equation*}
\|\nabla \times (u\times (\eps \Delta-\eps\mu)^{-1}B)\|_{L^p} \leq \frac{1}{2}\|B\|_{L^p}
\end{equation*} 
for all $B \in L^p_\omega(\Omega)$. Thus, for this value of $\mu$, $A_2-\eps \mu$ exists and is given by 
\begin{equation*}
(A_2-\eps\mu)^{-1}=(\eps \Delta-\eps\mu)^{-1} \sum_{n \geq 0}(-1)^n (\nabla \times (u \times (\eps \Delta-\eps \mu)^{-1}))^n.
\end{equation*}
We apply this operator now to $\nabla B$,
\begin{equation}
\label{eq:L resolvent expansion}
(A_2-\eps \mu)^{-1}\nabla B= (\eps \Delta-\eps \mu)^{-1} \nabla B+(\eps \Delta-\eps \mu)^{-1}\sum_{n \geq 0}(-1)^{n+1}(\nabla \times (u \times (\eps \Delta-\eps \mu)^{-1}))^n z,
\end{equation}
where $z=\nabla \times (u \times (\eps \Delta-\eps \mu)^{-1} \nabla B)$. In particular, it follows that $\|z\|_{L^p} \leq C\|B\|_{L^p}$, for some constant $C$ independent of $B \in L^p_\omega(\Omega)$. Hence, taking $W^{1,p}(\Omega)$ norms in \eqref{eq:L resolvent expansion}, there holds 
\begin{align*}
\|(A_2-\eps \mu)^{-1}\nabla B\|_{W^{1,p}} & \leq \|(\eps \Delta-\eps \mu)^{-1} \nabla B\|_{W^{1,p}} \\
& \quad +\|(\eps \Delta-\eps \mu)^{-1}\|_{L^p\to W^{1,p}}\sum_{n \geq 0}\|(\nabla \times (u \times (\eps \Delta-\eps \mu)^{-1}))\|^n_{L^p \to L^p}\|z\|_{L^p} \\
& \leq C\|B\|_{L^p},
\end{align*}
with $C$ independent of $B \in L^p_\omega(\Omega)$. In other words, the claim holds true for this particular value of $\mu$. But now, recalling the resolvent identity $R_\mu=R_s+(s-\mu)R_\mu R_s$, for $s,\mu \in \rho(A_2)$, we find
\begin{equation*}
\|(A_2-s)^{-1} \nabla B\|_{W^{1,p}} \leq \|(A_2-\eps\mu)^{-1}\nabla B \|_{W^{1,p}(\Omega)}+|\eps\mu-s| \|(A_2-\eps\mu)^{-1} (A_2-s)^{-1} \nabla B\|_{W^{1,p}}.
\end{equation*}
But we showed that $\|(A_2-\eps \mu)^{-1}B\|_{W^{1,p}} \leq C\|B\|_{L^p}$, so we conclude, using that by definition $(A_2-s)^{-1}:L^p(\Omega) \to L^p(\Omega)$ is bounded, that 
\begin{equation*}
\|(A_2-s)^{-1} \nabla B\|_{W^{1,p}}  
\leq C(\mu,s, \eps)\|B\|_{L^p},
\end{equation*}
for all $B \in L^p_\omega(\Omega)$, and the proof is complete.
\end{proof}

Finally, we need the following classical result on interpolation spaces, which may be found for instance in the monograph \cite{Hairer}.

\begin{lemma}[\cite{Hairer}, Proposition 4.45]
\label{lemma:hairer}
Let $A$ generate an analytic semigroup on some separable Banach space $X$, and suppose $S:(-A)^\alpha \to X$ is a bounded operator, for some $\alpha \in [0,1)$. Then, $A+S$ generates an analytic semigroup, and $\mathcal{D}(((-A-S))^\alpha) \sim \mathcal{D}((-A)^\alpha)$, for all $\alpha \in (0,1)$, i.e.\ the corresponding domains are equal, and the norms $\|(-A-S)^\alpha B\|$, $\|(-A)^\alpha B\|$ are equivalent for any $B \in \mathcal{D}((-A)^\alpha)$.
\end{lemma}

With this tool in our hand, we can now commence with the proof of the technical Lemma \ref{lemma:abstract semigroup lemma}.

\begin{proof}[Proof of Lemma \ref{lemma:abstract semigroup lemma}]
In view of the results in \cite{Friedlander_Pavlović_Shvydkoy_2006} for the linearised Navier--Stokes equations, it only remains to prove the analogous statement for $A_2$, i.e.\ the linearised magnetic operator. The first claim of Lemma \ref{lemma:abstract semigroup lemma} is true for any analytic semigroup, and so follows immediately from Lemma \ref{lemma:analytic semigroup}. Next, note that $\nabla \times (u_\TC \times \cdot)$ is a bounded operator from $W^{1,p}(\Omega)$ to $L^p(\Omega)$. Hence, for any $B \in \mathcal{D}((\eps \Delta-c_0)^{1/2})$, for any $c_0>0$ large enough there holds 
$$
\|\nabla \times (u_\TC \times B)\|_{L^p(\Omega)} \leq C\|B\|_{W^{1,p}(\Omega)} \leq C\|(\eps\Delta -c_0)^{1/2}B\|_{L^p},
$$
where the last inequality is from \eqref{eq:interpolation inequality}.
Hence, $S:=\nabla \times (u_\TC \times \cdot)$ is indeed bounded from $\mathcal{D}((\eps \Delta -c_0)^{1/2})$ to $L^p_\omega(\Omega)$. Thus, by Lemma \ref{lemma:hairer}, for all $\alpha \in [0,1]$ there exists $C(\alpha)$ so that 
\begin{equation*}
\|\Delta^\alpha B\|_{L^p} \leq C(\alpha)\|(A_2-\mu)^\alpha B\|_{L^p}.
\end{equation*}
Thus, let now $B \in L^q(\Omega)$ with $q \leq 3p/(3-(2\alpha-1)p)$ and $1/2<\alpha<1$. Using a density argument, we can assume that $B$ is smooth. Then, for $\mu \in \rho(A_2)$ it holds
\begin{equation}
\label{eq:final estimate 1}
\|(A_2-\mu)^{-\alpha} \nabla B\|_{L^q} \leq C\|(A_2-\mu)^{-\alpha} \nabla B\|_{W^{2\alpha-1,p}},
\end{equation}
from the Sobolev inequality. But then, we estimate 
\begin{equation}
\begin{split}
\label{eq:final estimate 2}
\|(A_2-\mu)^{-\alpha} \nabla B\|_{W^{2\alpha-1,p}} & = \|(A_2-\mu)^{1-\alpha}(A_2-\mu)^{-1} \nabla B\|_{W^{2\alpha-1,p}}\\
& \leq C \|(A_2-\mu)^{\alpha -1/2}(A_2-\mu)^{1-\alpha}(A_2-\mu)^{-1} \nabla B\|_{L^p} \\
& \leq  C\|(A_2-\mu)^{1/2}(A_2-\mu)^{-1} \nabla B\|_{L^p},
\end{split}
\end{equation}
where we used \eqref{eq:interpolation inequality} in the second to last inequality. But as noted already in \cite{Miyakawa_1980}, $\mathcal{D}(\Delta^{1/2}) \subset W^{1,p}(\Omega)$, with continuous injection, and by Lemma \ref{lemma:hairer}, the same holds for $A_2$. Thus, there exists a constant $C>0$ so that $\|(A_2-\mu)^{1/2}B\|_{L^p} \leq C\|B\|_{W^{1,p}}$. Hence, we deduce from \eqref{eq:final estimate 1} and \eqref{eq:final estimate 2} that 
\begin{equation}
\|(A_2-\mu)^{-\alpha} \nabla B\|_{L^q}\leq C\|(A_2-\mu)^{-1} \nabla B\|_{W^{1,p}} \leq C \|B\|_{L^p},
\end{equation}
where the last inequality follows by Lemma \ref{lemma:analytic semigroup}.
\end{proof}

\section{Proof of nonlinear instability}\label{s:proof}

We now have all the ingredients required to prove the claimed nonlinear instability result around the Taylor--Couette flow. We follow very closely the method of \cite{Friedlander_Pavlović_Shvydkoy_2006}, making the necessary adaptations to handle the addition of the magnetic field component.
\begin{proof}[Proof of Theorem \ref{theorem}]
Pick the growing mode $\phi_0$ of $A$ corresponding to the eigenvalue with largest real part, which we call $\lambda$. A growing mode $\phi_0$ of the form $(0,B_0)$ with $\Re(\lambda) \sim \eps^{1/3}>0$ is ensured to exist by Proposition \ref{prop:dynamo}. Via \eqref{eq:eq-perturbation} and the mild formulation, we write 
\begin{equation}\label{eq:w-proof}
w(t) = \begin{pmatrix}
    v(t)\\
    B(t)
\end{pmatrix} =\e^{A t}\delta\phi_0+\int_0^t \e^{A (t-\tau)}\begin{pmatrix}
-N(v,v)+N(B,B)\\
M(v,B)
\end{pmatrix}(\tau)
\dd \tau.
\end{equation}
We want to show that we the nonlinear instability mechanism is directly produced by the linear part of the formulation---that is the first term in the right hand side. In this regard, we start by looking at the nonlinear addend, that for some $q>1$ to be chosen later, can be estimated via Lemma \ref{lemma:abstract semigroup lemma} by
\begin{equation*}
\begin{split}
& \left \|\int_0^t \e^{A (t-\tau)}\begin{pmatrix}
-N(v,v)+N(B,B)\\
M(v,B)
\end{pmatrix}(\tau)
\dd \tau \right \|_{L^q}\\
& \qquad \leq \int_0^t \e^{(\lambda+\eta)(t-\tau)} \left\|A_\eta^\alpha e^{A_\eta (t-\tau)} A^{-\alpha}_\eta \begin{pmatrix}
-N(v,v)+N(B,B)\\
M(v,B)
\end{pmatrix}\right\|_{L^q} \dd \tau\\
& \qquad \lesssim \int_0^t \e^{(\lambda+\eta)(t-\tau)} \frac{1}{(t-\tau)^\alpha} \left \| A_\eta^{-\alpha} \begin{pmatrix}
-N(v,v)+N(B,B)\\
M(v,B)
\end{pmatrix}\right \|_{L^q} \dd \tau.
\end{split}
\end{equation*}
Moreover, by letting $\alpha\in (1/2,1)$ and $q>3/(2\alpha-1)$, we can use the second claim from Lemma \ref{lemma:abstract semigroup lemma} to write
\[
\|A_\eta^{-\alpha}N(v,v)\|_{L^q} \leq \|A_\eta^{-\alpha}\nabla (v \otimes v)\|_{L^q}\lesssim \|v\otimes v\|_{L^{q/2}} \lesssim \|v\|_{L^q}^2.
\]
Analogously, we bound $\|A^{-\alpha}_\eta N(B,B)\|_{L^q} \lesssim \|B\|_{L^q}^2$, and $\|A^{-\alpha}_\eta M(v,b)\|_{L^q} \lesssim \|B\|_{L^q}\|v\|_{L^q}$, which in particular yields the estimate
\[
\left \|\int_0^t \e^{A (t-\tau)}\begin{pmatrix}
-N(v,v)+N(B,B)\\
M(v,B)
\end{pmatrix}(\tau)
\dd \tau \right \|_{L^q} \lesssim \int_0^t \e^{(\lambda+\eta)(t-\tau)} \frac{1}{(t-\tau)^\alpha}\|w(\tau)\|_{L^q}^2 \dd \tau.
\]
For any $Q >\|\phi_0\|_{L^q}$, we now define a time $T\in (0,\infty]$ given by
\begin{equation}\label{eq:maximalT}
T=\inf\{t \geq 0: \|w(t)\|_{L^q}\geq \e^{\lambda t}\delta Q\}.
\end{equation}
Hence, for any $t \leq T$, there holds 
\begin{equation*}
\int_0^t \e^{(\lambda+\eta)(t-\tau)} \frac{1}{(t-\tau)^\alpha}\|w(\tau)\|_{L^q}^2 \dd \tau \leq Q^2\delta^2\int_0^t\e^{(\lambda+\eta)(t-\tau)} \frac{1}{(t-\tau)^\alpha} \e^{2\lambda \tau } \dd \tau \lesssim Q^2 \delta^2 \e^{2\lambda t},
\end{equation*}
for all $t \leq T$, provided that we choose $\eta$ small enough. All in all, we found that there exists a constant $C>0$, such that the $L^q$ norm of the nonlinear part is bounded by
\[
\left \|\int_0^t \e^{A (t-\tau)}\begin{pmatrix}
-N(v,v)+N(B,B)\\
M(v,B)
\end{pmatrix}(\tau)
\dd \tau \right \|_{L^q} \leq CQ^2 \delta^2 \e^{2\lambda t}
\]
for any $t\leq T$. This result holds true for $q>3/(2\alpha-1)$, $\alpha\in(1/2,1)$, which in turn implies that $q>3$. However, in order to obtain a similar estimate for all $p>1$, we can simply make use of the fact that the domain is bounded to write $\|\cdot \|_{L^p}\leq C'\|\cdot \|_{L^q}$.
Going back to \eqref{eq:w-proof} and using the definition of $T$ in \eqref{eq:maximalT}, we find that at time $t=T$ there holds
\[
\delta Qe^{\lambda T}=\|w(T)\|_{L^q} \leq \|\phi_0\|_{L^q}\delta e^{\lambda T}+CQ^2 \delta^2 e^{2\lambda T}.
\]
Here we are formally assuming that $T<\infty$, but the inequality holds true for any $t\leq T$ by \eqref{eq:maximalT}. Thus, we find that there exists a constant $\chi_\star>0$ with the property
\begin{equation*}
\delta \e^{\lambda T} \geq \frac{Q-\|\phi_0\|_{L^q}}{CQ^2} =: \chi_\star.
\end{equation*}
Therefore, there must exist some $t_* <T$ so that $\delta e^{\lambda t_*}=\chi_\star$. 
Evaluating \eqref{eq:w-proof} at time $t=t_\star$, taking the $L^p$ norm and using the triangle inequality we can write
\begin{equation*}
\|w(t_\star)\|_{L^p} \geq \delta e^{\lambda t_\star}\|\phi_0\|_{L^p} - C'CQ^2 \delta^2 e^{2\lambda t_\star} \geq \chi_\star\|\phi_0\|_{L^p} - CC'Q^2 \chi_\star^2 \geq \chi_\star(\|\phi_0\|_{L^p}-CC'Q^2 \chi_\star).
\end{equation*}
All in all, picking $Q$ sufficiently close to $\|\phi_0\|_{L^q}$, we deduce that 
\begin{equation*}
\|w(t_\star)\|_{L^p} \geq \frac{1}{2}\chi_\star,
\end{equation*}
completing the proof. To deduce that the initial perturbation may be in $W^{s,p}$, note that by elliptic regularity, any eigenfunction of $A$ will have $W^{k,p}$ norm comparable to its $L^p$ norm, and so the result follows immediately by the above.

Furthermore, we readily see that if the growing mode $\phi_0$ of $A$ corresponding to the eigenvalue with largest real is of the form $\phi_0=(0,B_0)$, then the growth occurs for the magnetic component and we can write
\[
\|B(t_\star)\|_{L^p} \gtrsim \|w(t_\star)\|_{L^p}\geq \frac{1}{2}\chi_\star
\]
as claimed. In particular, via Proposition \ref{prop:dynamo} we see that this will be the case if $\nu>0$ is large when compared to $\|u_\TC\|_{W^{1,\infty}}$, since via Poincaré we can write
\[
\|A_1v\|_{L^2} \leq - (C\nu-\|u_\TC\|_{W^{1,\infty}})\|v\|_{L^2}
\]
for some constant $C>0$. Since any eigenfunction of $A_1$ in $L^p$ must as well be in $L^2$, we conclude that no growing mode can exist for $A_1$ in $L^p$ if $\nu\gg \|u_\TC\|_{W^{1,p}}$.
\end{proof}

\addtocontents{toc}{\protect\setcounter{tocdepth}{0}}
\section*{Acknowledgements}
The authors thank Rajendra Beekie, Niklas Knobel and Michele Coti Zelati for insightful comments.
The research of VNF is funded by the ERC-EPSRC Horizon Europe Guarantee EP/X020886/1. The research of DV is funded by the Imperial College President's PhD Scholarships.

\addtocontents{toc}{\protect\setcounter{tocdepth}{1}}
\bibliographystyle{abbrv}
\bibliography{dynamo.bib}
\end{document}